\documentclass[reqno]{amsart}
\usepackage[psamsfonts]{amssymb}
\usepackage{amsthm}
\usepackage{color}

\newtheorem{thm}{Theorem}

\newtheorem{lem}{Lemma}
\newtheorem{cor}{Corollary}

\theoremstyle{remark}

\theoremstyle{definition}

 \theoremstyle{definition}

 \theoremstyle{remark}

 \numberwithin{equation}{section}

\newcommand{\NN}{{\mathbb N}}

\def\bege{\begin{equation}} \def\ende{\end{equation}}

   \def\begr{\begin{eqnarray}}
\def\endr{\end{eqnarray}} 
 
\def\bege{\begin{equation}} \def\ende{\end{equation}}
\def\begr{\begin{eqnarray}} \def\endr{\end{eqnarray}}
\def\bnum{\begin{enumerate}} \def\enum{\end{enumerate}}

\allowdisplaybreaks[4]

\begin{document}
\title{\sf Vanishing Carleson measures and power compact  weighted composition operators }
\thanks{The first author is thankful to  CSIR(India), Grant Number (File no. 09/1231(0001)/2019-EMR-I).\\
The second author is  thankful to  DST(SERB) for the project grant under MATRICS scheme. Grant No. MTR/2018/000479.}

\keywords { weighted composition operator, Carleson measure, power bounded, power compact and admissible weight.}
\subjclass[2020 MSC]{ 47B38, 47B33, 46E10, 46B50}

\author{Aakriti
 Sharma}

\address{Department of Mathematics,
Central University of Jammu,
Bagla,
Rahya-Suchani, Samba 181143,
INDIA }
\email{aakritishma321@gmail.com}

\author{Ajay K.~Sharma}
\address{Department of Mathematics,
Central University of Jammu,
Bagla,
Rahya-Suchani, Samba 181143,
INDIA }
\email{aksju\_76@yahoo.com}

\author{M. Mursaleen}
\address{Department of Mathematics, Aligarh Muslim University, Aligarh 202002,  India}
\email{mursaleenm@gmail.com}

\maketitle

\begin{abstract}
 In this paper, we characterize Carleson measure and vanishing Carleson measure on Bergman spaces with admissible weights in terms of {\it t-Berezin transform}  and  {\it averaging function}  as key tools. Moreover, power bounded and power compact weighted composition operators are characterized as application of Carleson measure and vanishing Carleson measure respectively on Bergman spaces with admissible weights.\\

 \end{abstract}

\section{Introduction}
\label{intro}
 Let $\mathcal{H}(\mathbb{D}) $ denote the space of analytic functions on the unit disk $\mathbb{D}=\{z{\in} \mathbb{C}:|z|<1\}$. Given a positive integrable function $\sigma\in C^{2}[0,1]$, we extend it on $\mathbb D$ by defining $\sigma(z)=\sigma(|z|), z\in \mathbb{D}$ and call such $\sigma$ a weight function.\\
 For $0<p<{\infty}$ and a positive Borel measure $\Omega$, the space $L^{p}({\Omega})$ consists of all measurable functions $f$ on $\mathbb{D}$ for which $$\|f\|^p_{L^{p}({\Omega})}=\int_{\mathbb{D}}|f(z)|^{p}d\Omega(z)<\infty .$$
 In the case $p=\infty$, the space of all complex-valued measurable functions $f$ on $\mathbb{D}$ is defined as $$L^{\infty}({\Omega})=\{f\in\mathcal{H}(\mathbb{D}):\|f\|_{\infty}={\rm ess\: sup}_{z \in \mathbb D}|f(z)|<\infty\}, $$ where the essential supremum is taken with respect to the measure $\Omega.$
 A sequence $\{f_n\}_{n\in\NN}$ is norm bounded in $L^{\infty}({\Omega})$ if ${\rm sup}_{n \in \mathbb N}\|f_n\|_{\infty}$ is finite. Let $dA(z)=\frac{dxdy}{\pi}$ be the normalized Lebesgue area measure on $\mathbb{D}$,
 we define the weighted Bergman space as $$A_{\sigma}^{p}=\{f\in\mathcal{H}(\mathbb{D}):\|f\|_{A_{\sigma}^{p}}^{p}
 =\int_{\mathbb{D}}|f(z)|^{p}\sigma{(z)}dA(z)<\infty\} .$$
 Note that $A_{\sigma}^{2}$ is a closed subspace of $L^2(\sigma dA)$ and hence is a Hilbert space
endowed with the inner product $$\langle f,g\rangle=\int_{\mathbb{D}}f(z)\overline{g(z)}\sigma(z)dA(z)\;\;\;\;f,g\in A_{\sigma}^{2}.$$\\
 Throughout this paper, we will consider $\sigma$ as admissible weight function. Recall that if a weight function $\sigma$ is non-increasing on $[0,1)$ and $\sigma(r)(1-r)^{-(1+\delta)}$ is non decreasing on $[0,1)$ for some $\delta>0$, then $\sigma$ is called admissible weight.\\
We refer the readers \cite{ZK} for useful fact over pseudohyperbolic metric. The pseudohyperbolic metric is defined as $\rho(a,z)= |\phi_{a}(z)|$, where $\phi_{a}(z)= \frac{a-z}{1-\bar{a}z}$ is m\"{o}bius transformation. For $r$ in $(0,1)$ and $a$ in $\mathbb{D}$, $E(a,r)=\{z\in \mathbb{D}:\rho(z,a)<r\}=\phi_a(E(0,r))=\phi_a(\{z:|z|<r\})$ denote the pseudohyperbolic disk centre at $a$ and radius $r\in(0,1)$. It turns out $E(a,r)$ is a Euclidean disk with center $\dfrac{1-r^{2}}{1-|z|^{2}r^2}z$ and radius $\dfrac{(1-|z|^{2})r}{1-|z|^{2}r^2}$.\\
 For every $z\in E(a,r)$, $(1-|a|^{2})^{2}\asymp (1-|z|^{2})^{2}\asymp |{1-\bar{a}z}|^{2}$ and area of $E(a,r)$ is denoted by $|E(a,r)|$ and $|E(a,r)|\asymp (1-|a|^{2})^{2}$ are well known facts. Here the symbol $"\asymp"$ denotes that the left hand side is bounded above and below by constant multiples of the right hand side, where the constants are positive and independent of variables. Given $r$ in $(0,1)$, a sequence $\{z_{k}\}_{k=1}^{\infty}\subset\mathbb{D}$ is said to be an \textit{r-lattice} if the disk $\{E(z_k,r)\}_{k=1}^{\infty}$ cover $\mathbb{D}$ and there is some integer $M>0$ such that each $z$ in $\mathbb{D}$ belongs to at most $M$ of the disks $\{E(z_k,\frac{1+r}{2})\}_{k=1}^{\infty}$. Equivalently, \begin{equation}\label{a1}1\leq\sum_{k=1}^{\infty}\chi_{E(z_{k},\frac
{1+r}{2})}(z)\leq M.\end{equation}
Recall that, $$K^{\alpha}(z,w)=\dfrac{1}{(1-\bar{w}z)^{\alpha+2}},\;\;\;z\in \mathbb{D}$$ is the reproducing kernel in $\textit{standard weighted Bergman space}$ $A_{\sigma}^{p}=A_{\alpha}^{p}$, where standard weight $\sigma(z)=(1-|z|^{2})^{\alpha+2}$, $\alpha>-1$. The normalized reproducing kernel of $A_{\alpha}^{p}$ is defined as $$k^{\alpha}_{z}(.)=\dfrac{K^{\alpha}(z,.)}{\sqrt{K^{\alpha}(z,z)}}$$ where $K^{\alpha}(z,z)=\dfrac{1}{(1-|z|^{2})^{\alpha+2}}$.\\
  For a finite positive Borel measure $\Omega$ on $\mathbb{D}$, the \textit{t-Berezin transform} is defined to be  \begin{equation}\label{a2} \widetilde{\Omega}_{t}(z)=\int_{\mathbb{D}}\left(|k_z^{\alpha}(w)|\right)^{t}d\Omega(w),\;\;\;z\in \mathbb{D}\end{equation}
Note that for $t=2$, the classical Berezin transform is denoted by $\widetilde{\Omega}_{2}$. Given $r$ in $(0,1)$, the Averaging function of $\Omega$ is defined to be\begin{equation}\label{a3}\widehat{\Omega}_{r}(z)=\frac{\Omega(E(z,r))}
{|E(z,r)|},\;\;\;z\in \mathbb{D}\end{equation}
If we set $d\Omega=fdA$, for a Lebesgue measurable function $f$, then we can write $\widetilde{f_{t}}=\widetilde{\Omega}_{t}$ and $\widehat{f}_{r}=\widehat{\Omega}_{r}$ for simplicity.\\
Motivated by \cite{HLZ} and \cite{KH}, in this article we study the power bounded and power compact weighted composition operator on $A_{\sigma}^{p}$ by using Carleson measure characterization in terms of $\textit{t-Berezin transform}$ and $\textit{averaging function}$ as a key tool. An operator $T$ on a normed linear space $(X,\|.\|_{X})$  is called ${\textit{power bounded}}$ if $\{T^{n}\}$ is a bounded sequence in the space of all bounded operators from $X$ to itself. Also, recall that an operator $T$ on Banach space $(X,\|.\|_{X})$ is said to be ${\textit{power compact}},$ see \cite{DGT} if there exist some integer $m>0$ such that $T^{m}$ is compact from $X$ to itself. Denote by $\Lambda^2(\mathbb C)$, the linear space   of all double sequences with complex entries.  A double sequence $\{\gamma_{j, k}\}_{j,  k \in \mathbb N}$ of complex numbers  is bounded if there exists some $M > 0$ such that $\sup_{j, k }|\gamma_{j,k}| \leq M$. The space $\Lambda^2_\infty $  of all  bounded
double sequences is defined as
$$\Lambda^2_\infty = \{\gamma_{jk} = \{\gamma_{j, k}\}_{j,  k \in \mathbb N} \in \Lambda^2(\mathbb C) : \|\gamma_{jk}\|_{\Lambda^2_\infty} = \sup_{j, k }|\gamma_{j,k}| < \infty\}.$$
 \\
Let $C_{\psi,\phi}$ denoted the well known $\textit{weighted composition operator}$ on the space $\mathcal{H}(\mathbb{D})$ is defined as $$C_{\psi,\phi}(f)=\psi {(f\circ{\phi})} $$
where $\psi \in \mathcal{H}(\mathbb{D})$ and $\phi$ is an analytic self map of $\mathbb{D}$. If $\phi(z)=z$ and $\psi=1$, then $C_{\psi,\phi}$ becomes the multiplication operator $M_{\psi}$ and the composition operator $C_{\phi}$ respectively. Denote by $\phi_n$ the $n$th iteration of $\phi,$ that is $$\phi_n = \underbrace{\phi \circ \phi \cdots \phi}_{n-\text{times}}. $$
Note that any power of $C_{\psi,\phi}$ on $\mathcal{H}(\mathbb{D})$ is a weighted composition operators which is defined as
   $$ C_{\psi,\phi}^{n}f =\prod_{j=0}^{n-1}(\psi\circ \phi_{j})f\circ \phi_n.$$ For the sake of simplicity, we set $$ \langle\psi,\phi,n\rangle=\prod_{j=0}^{n-1}\psi\circ \phi_{j} . $$ Thus, $$\|C_{\psi,\phi}^{n}f\|_{L^{p}({\Omega})}^{p}=\int_\mathbb{D}|\langle\psi,\phi,n\rangle(z)f\circ \phi_n(z)|^{p}d\Omega(z).$$ We define $d\Omega_{n}=|\langle\psi,\phi,n\rangle |d\Omega\circ\phi^{-1}_{n}$,   one can easily see that $\Omega_n$ is a measure and    therefore $$\|C_{\psi,\phi}^{n}f\|_{L^{p}({\Omega})}^{p}=\int_\mathbb{D}|f(z)|^{p}d\Omega_{n}$$.For $t>0$, the $t$-Berezin transform and for $0<r<1$, the averaging function of $\Omega_n$ are defined as $$\widetilde{\Omega}_{n,t}(z)=\int_\mathbb{D}\big(|k_z^{\alpha}(w)|\big)^{t}d\Omega_n(z),\;\;\;z\in\mathbb{D}$$
and $$\widehat{\Omega}_{n,r}(z)=\dfrac{\Omega_n( E(z,r))}{|E(z,r)|},\;\;\;z\in\mathbb{D}$$ respectively. For more about weighted composition operators, Carleson measures and vanishing Carleson measures, we refer to \cite{MMD}-\cite{ZK}.
Throughout the paper, the expression $E\lesssim F$ means that there exists a constant $C$ such that $E\le CF$.

\section{Preliminaries}

In this section, we  prove and  collect some useful facts  and lemmas  that are required for the proof of our main results.
 The next lemma gives a growth estimate for functions in $A_{\sigma}^{p}$ and an asymptotic estimate for norm of $K^{\alpha}(z, \cdot)$ already proved in \cite{SAU}.
\begin{lem}\label{lem1} Let $p>0$, $\sigma$ be an admissible weight and $\alpha > -1$. Then
\begin{enumerate}
\item  for each $z\in \mathbb{D},$ we have that
\begin{equation} |f(z)|^{p} \lesssim \dfrac{\|f\|_{A_{\sigma}^{p}}^{p}}{\sigma(z)(1-|z|^{2})^{2}} \; \text{ for all }  \;f \in A_{\sigma}^{p}.
\end{equation}
\item  for each $z\in \mathbb{D},$ we have that
\begin{equation} \|K^{\alpha}(z, \cdot)\|_{A_{\sigma}^{p}} \asymp \dfrac{1}{(\sigma(z))^{\frac{1}{p}}(1-|z|^{2})^{(\alpha +2)-\frac{2}{p}}}
\end{equation}
\end{enumerate}
\end{lem}

\begin{lem} \label{lem2}Let $1\leq p\leq{\infty}$ and $r\in(0,1)$. Then the operator $f\longmapsto\widehat{f}_{r}$ is bounded from $L^{p}(\mathbb{D})$ to $L^{p}(\mathbb{D})$.\end{lem}
\begin{proof}Since $\chi_{E(z,r)}(\xi)=\chi_{E(\xi,r)}(z),\;\;z,\xi\in \mathbb{D}$. By Fubini's theorem for all $f\in L^{1}(\mathbb{D})$, we have
\begin{align}\|\widehat{f}_{r}\|_{1}&\leq \int_{\mathbb{D}}\bigg|\frac{1}{|E(z,r)|}\int_{E(z,r)}|f(\xi)|dA(\xi)\bigg| dA(z)\notag \\
&\leq \int_{\mathbb{D}}\frac{dA(z)}{|E(z,r)|}\int_{E(z,r)}|f(\xi)|dA(\xi)\notag\\ &=\int_{\mathbb{D}}|f(\xi)|dA(\xi)\int_{E(\xi,r)}\frac{dA(z)}{|E(z,r)|}\notag\\ &\leq C\|f\|_{1}\notag
\end{align}
The boundedness of the operator $f\longmapsto\widehat{f}_{r}$  is trivially holds for $p=\infty$ , that is $\|\widehat{f}_{r}\|_{\infty}\leq \|f\|_{\infty}$ and also holds for $1< p< \infty$, by complex interpolation.
\end{proof}

\begin{lem} \label{lem3}Suppose $0<p<{\infty}$, $r\in (0,1)$ and $\Omega$ be a finite positive Borel measure. Then there exists some constant $C$ such that $$\int_{\mathbb{D}}h(z)d\Omega(z)\le C\int_{\mathbb{D}}h(z)\widehat{\Omega}_{r}(z)dA(z) $$for all non-negative subharmonic function $h:\mathbb{D}\longrightarrow [0,\infty)$.
\end{lem}
\begin{proof}Since $h$ be non negative subharmonic function $h:\mathbb{D}\longrightarrow [0,\infty)$. Then,  we have $$h(z)\le \frac{C}{|E(z,r)|}\int_{E(z,r)}h(\xi)dA(\xi)$$\;\;\; for all $z\in \mathbb{D}$,  see page no. 125 of \cite{ZK}. Using above inequality, Fubini's theorem and (\ref{a3}), we have that \begin{align*}\int_{\mathbb{D}}h(z)d\Omega(z)&\le C\int_{\mathbb{D}} \frac{d\Omega(z)}{|E(z,r)|}\int_{E(z,r)}h(\xi)dA(\xi)\\ &= C\int_{\mathbb{D}}h(\xi)dA(\xi)\int_{E(\xi,r)}\frac{d\Omega(z)}{|E(z,r)|}\\
&\le C\int_{\mathbb{D}}h(z)\widehat{\Omega}_{r}(z)dA(z)
\end{align*}
This accomplished the result.\end{proof}

Note that lemma \ref{l4} and lemma \ref{lem4} already prove in \cite{HLZ} and \cite{HZL} in more general sense. For the sake of completeness, we include the proofs. Note that lemma \ref{l4} reveals that the function $\Omega(E(z,r))$ behaves like a subharmonic function. For $r=R$, the following lemma, lemma \ref{l4} was first obtained in \cite{ZK1}.

\begin{lem}\label{l4} Suppose $r,R >0$ and $\Omega$ be a positive borel measure on $\mathbb{D}$, then there exists a constant $C$ such that
\begin{equation}\label{ee}
  \Omega(E(a,R))\le \dfrac{C}{|E(a,R)|}\int_{E(a,R)}\Omega(E(z,r))dA(z)
\end{equation}
\end{lem}
\begin{proof}
For $r,R>0$, we have
\begin{align*}\label{aa}
  \int_{E(a,R)}\Omega(E(z,r))dA(z) &= \int_{\mathbb{D}}\chi_{E(a,R)}(z)dA(z)
  \int_{\mathbb{D}}\chi_{E(z,r)}(w)d\Omega(w)\\
  &=\int_{\mathbb{D}}d\Omega(w)\int_{E(a,R)}\chi_{(E(z,r)}(w)dA(w)
\end{align*}

Since $\chi_{E(z,r)}(w)=\chi_{E(w,r)}(z)$ for all $z,w\in \mathbb{D}$, then
\begin{align*}
  \int_{E(a,R)}\Omega(E(z,r))dA(z) &= \int_{\mathbb{D}}d\Omega(w)\int_{E(a,R)}\chi_{E(w,r)}(z)dA(z) \\
 &\geq \int_{E(a,R)} |(E(a,R)\cap E(w,r))|d\Omega(w)  \\
  &\geq \Omega(E(a,R))\inf_{w\in E(a,R)}\{|(E(a,R)\cap E(w,r))|\}
\end{align*}
For $w\in (E(a,R)$, then there exists a Euclidean disk with diameter $\frac{1}{2}\min\{r,R\}$ contained in $(E(a,R)\cap E(w,r))$. Therefore (\ref{ee}) holds.

\end{proof}

\begin{lem} \label{lem4} Suppose $1 \le p \le {\infty}$ and $s\in \mathbb{R}$. Then $\dfrac{\widehat{\Omega}_{r}(z)}{\sigma^{s}(z)\left(1-|z|^{2}\right)^{2(s-1)}}\in L^{p}(\mathbb{D})$ for some $r$, $0<r<1$ if and only if $\dfrac{\widehat{\Omega}_{R}(z)}{\sigma^{s}(z)\left(1-|z|^{2}\right)^{2(s-1)}}\in L^{p}(\mathbb{D})$ for all $0<R<1$. Moreover, $\bigg \|\dfrac{\widehat{\Omega}_{r}(z)}{\sigma^{s}(z)\left(1-|z|^{2})\right)^{2(s-1)}}\bigg \|_{p}\asymp \bigg\|\dfrac{\widehat{\Omega}_{R}(z)}{\sigma^{s}(z)\left(1-|z|^{2}\right)^{2(s-1)}}\bigg \|_{p}$.
\end{lem}
\begin{proof} For $r$ and $R$, by lemma \ref{l1} there exists a constant $C$ such that $$\widehat{\Omega}_R(z)\le \dfrac{C}{|E(z,R)|}\int_{E(z,R)}\widehat{\Omega}_{r}(\xi)dA(\xi)$$ for all $z\in \mathbb{D}$. Above inequality and (\ref{a3}) implies that, \begin{align*}G_R(z) &\le\dfrac{C}{|E(z,R)|}
\int_{E(z,R)}G_r(\xi)dA(\xi)\\
&\asymp\widehat{G_{r}}(z), \end{align*}
where $G_l(z)=\frac{\widehat{\Omega}_{l}(z)}{\sigma^{s}(z)\left(1-|z|^{2}\right)^{2(s-1)}}$. Combining this with Lemma \ref{lem3}, we find that
$$\bigg \|\frac{\widehat{\Omega}_{R}(z)}{\sigma^{s}(z)\left(1-|z|^{2}\right)^{2(s-1)}}\bigg \|_{p}\le  \bigg \|\widehat{G_{r}}(z)\bigg \|_{p} \le\bigg \|G_r(z)\bigg \|_{p}\;\;\;\;\text{for}\; 1\le p\le \infty .$$
Interchanging the role of $r$ and $R$, we get the desired result.
\end{proof}

\begin{lem} \label{lem5}Suppose $1<p< \infty$ and $t>0$. Then the integral operator $$T_{t,s}f(z)=\sigma^{s}(z)\big(1-|z|^{2}\big)^{(\alpha +2)t-2s}\int_{\mathbb{D}}\dfrac{|K^{\alpha}(z,\xi)|^{t}}{\sigma^{s}(\xi)
\left(1-|\xi|^{2}\right)^{2(s-1)}}f(\xi)dA(\xi)$$ is bounded on $L^{p}(\mathbb{D})$ whenever $\frac{1-p}{p}<s<(\alpha +2)\frac{t}{2}+\frac{1}{p}$.
\end{lem}
\begin{proof}
Let $P=\frac{1-p}{p}(s+1)$, $Q= \frac{p-1}{p}\{(\alpha +2)t - {2}-s\}$, $U=\frac{s+1-(\alpha +2)t}{p}$ and $V= \frac{s}{p}$. Then claim that the intervals $(P,Q)$ and $(U,V)$ are non-empty. By using hypothesis one can easily find that
\begin{align*}
  Q-P &=(\alpha+2)(1-\frac{1}{p})(t-\frac{1}{\alpha+2}) >0
\end{align*}
and \begin{align*}
         V-U &=\frac{(\alpha+2)}{p}(t-\frac{1}{\alpha+2})>0
         \end{align*}
          Implies that the intervals $(P,Q)$ and $(U,V)$ are non-empty. Also,  \begin{align*}
                                           V-P &=s-\frac{1-p}{p}>0
                                         \end{align*}
          and \begin{align*}
                  Q-U &=(\alpha+2)t-2+\frac{1}{p}-s>0
                \end{align*}
          implies that $P<V$ and $U<Q$. Thus, $(P,Q)\cap (U,V)$ is non-empty.
  For some $m\in (P,Q)\cap (U,V)$ and  take $h(\xi)=\big(\sigma(\xi)(1-|\xi|^{2})\big)^{m}$ and
 $\frac{1}{p}+\frac{1}{p^{'}}=1$. From Lemma  \ref{lem1}, there is a positive constant $C$ such that $$\sigma^{s}(z)\big(1-|z|^{2}\big)^{(\alpha +2)t-2s}\int_{\mathbb{D}}\dfrac{|K^{\alpha}(z,\xi)|^{t}}{\sigma^{s}(\xi)
 \left(1-|\xi|^{2}\right)^{2(s-1)}}h^{p^{'}}(\xi)
 dA(\xi)\le Ch^{p^{'}}(\xi),\;\;\; z\in \mathbb{D}$$ and
$$\sigma^{-s}(z)\left(1-|z|^{2}\right)^{2(1-s)}\int_{\mathbb{D}}|K^
{\alpha}(z,\xi)|^{t}\sigma^{s}(\xi)\big(1-|\xi|^{2}
\big)^{(\alpha +2)t-2s}h^{p}(\xi)dA(\xi)\le C h^{p}(\xi)$$ $z\in \mathbb{D}$. By Schur's test, boundedness of the operator
$T_{t,s}$ on $L^{p}(\mathbb{D})$ holds.
\end{proof}

\begin{lem}\label{lem6}Let $\{z_{k}\}_{k=1}^{\infty}$ be an $r$- lattice. For $1<p< \infty$ and $\{\lambda_{k}\}_{k=1}^{\infty}\in l^p$, let \begin{equation}\label{aa1} f(z)= \sum_{k=1}^{\infty}\lambda_{k}\dfrac{K^{\alpha}({z_k},z)}{\sigma^
{\frac{1}{p}}(z_k)(1-|z_k|^{2})^{\frac{2}{p}-(\alpha +2)p}}\end{equation} where $\alpha>-1$. Then $f\in A_{\sigma}^{p}(\mathbb{D})$ and $\|f\|_{p,\sigma}\le C \|\{\lambda_k\}_{k=1}^{\infty}\|_{l^p}$.
\end{lem}
The proof is an easy modification of arguments in Theorem 4.1 in \cite{SAU}. We omit the details.

\begin{lem} \label{lem7}Suppose $\Omega\ge 0$, $1\le p\le \infty$, $t>0$ and $s\in \mathbb{R}$ satisfies $t<s+\frac{1}{p}<1$. Then the following are equivalent
\begin{enumerate}\item[{(a)}] $\widetilde{M}_{t,s}(z)=\dfrac{\widetilde{\Omega}_{t}(z)}{\sigma^{s}(z)(1-|z|^{2})^{2s-(\frac{\alpha+2}{2})t}}\in L^{p}(\mathbb{D})$.
\item[{(b)}] $\widehat{M}_{R,s}(z)=\dfrac{\widehat{\Omega}_{R}
    (z)}{\sigma^{s}(z)(1-|z|^{2})^{2(s-1)}}\in L^{p}(\mathbb{D})$ for some (or any) $R$, $0<R<1$.
\item[{(c)}]The sequence $\bigg\{\dfrac{\widehat{\Omega}_{r}(z_{k})}
    {\sigma^{s}(z_{k})(1-|z_{k}|^{2})^{2(s-1-\frac{1}{p})}}\bigg\}_{k=1}^{\infty}$ belongs to $l^{p}$ for some (or any) $r$-lattice $\{z_{k}\}_{k=1}^{\infty}$ with $0<r<1$.
\end{enumerate}
 Moreover, we have $$\|\widetilde{M}_{t,s}\|_{p}\asymp \|\widehat{M}_{R,s}\|_{p}\asymp \bigg\|\bigg\{\dfrac{\widehat{\Omega}_{r}
 (z_{k})}{\sigma^{s}(z_{k})(1-|z_{k}|^{2})^
 {2(s-1-\frac{1}{p})}}\bigg\}_{k=1}^{\infty}\bigg\|_{l^{p}}.$$
\end{lem}
\begin{proof} We will prove the result in the order: $(a)\Leftrightarrow(b)$ and $(b)\Leftrightarrow (c).$\\
$(a)\Rightarrow (b) $. For any $R\in(0,1)$, there exist positive constant $C_R$ such that for any $z\in \mathbb{D}$ kernel estimate holds. Thus for $s\in \mathbb{R}$ , we have \begin{align*} \widehat{M}_{R,s}(z) &=\dfrac{\Omega(E(z,R))}{\sigma^{s}(z)(1-|z|^{2})^{2s}}\\& \le C_{R} \dfrac{1}{\sigma^{s}(z)(1-|z|^{2})^{2s-(\frac{\alpha +2}{2})t}}
\int_{E(z,R)}
|k_{z}^{\alpha}(\xi)|^{t}d\Omega(\xi) \\ &\le C_{R} \dfrac{1}{\sigma^{s}(z)(1-|z|^{2})^{2s-(\frac{\alpha +2}{2})t}}
\int_{\mathbb{D}}
|k_{z}^{\alpha}(\xi)|^{t}d\Omega(\xi) \\ & \le C_{R}\widetilde{M}_{t,s}(z)
\end{align*}
Above implies that, $\|\widehat{M}_{R,s}\|_{p}\le C\|\widetilde{M}_{t,s}\|_{p}$.\\

$(b)\Rightarrow (a)$. By Lemma  \ref{lem2}  and Lemma \ref{lem4}, there is a positive constant $C$ such that for any $z\in \mathbb{D}$ and $s\in \mathbb{R}$, we have \begin{align*}\widetilde{M}_{t,s}(z)&=\dfrac{\widetilde{\Omega}_{t}(z)}{\sigma
^{s}(z)(1-|z|^{2})^{2s-(\frac{\alpha+2}{2})t}}\\&=\dfrac{1}{\sigma^{s}(z)
(1-|z|^{2})^{2s-(\frac{\alpha +2}{2})t}}\int_{\mathbb{D}}
|k_{z}^{\alpha}(\xi)|^{t}d\Omega(\xi)\\&\le \dfrac{C}{\sigma^{s}(z)(1-|z|^{2})^{2s-(\frac{\alpha +2}{2})t}}\int_{\mathbb{D}}
|k_{z}^{\alpha}(\xi)|^{t}\widehat{\Omega}_{R}
(\xi)dA(\xi)\\&\le \dfrac{C}{\sigma^{s}(z)(1-|z|^{2})^{2s-(\alpha +2)t}}\int_{\mathbb{D}}
|K^{\alpha}(z,\xi)|^{t}\sigma^s(\xi)(1-|\xi|^{2})^{2(s-1)}\widehat{M}_{R,s}
(\xi)dA(\xi)\\&\le CT_{t,-s}(\widehat{M}_{R,s})(z)
\end{align*}
Since $t<s+\frac{1}{p}<1$ and Lemma \ref{lem5}  implies that $$\|\widetilde{M}_{t,s}\|_{p}\le \|T_{t,-s}(\widehat{M}_{R,s})\|_{p}\le C\|\widehat{M}_{R,s}\|_{p}.$$

$(b)\Rightarrow (c)$. Assume that $\widehat{M}_{R,s}\in L^{p}(\mathbb{D})$ for some $R$, $0<R<1$. Let $\{z_{k}\}_{k=1}^{\infty}$ be any $r$-lattice. By Lemma \ref{lem4}, we may assume $R<r$. By triangle inequality, we have $E(z_{k},r)\subset E(z,2r)$ for $z\in E(z_{k},r)$ and for all $k$. Thus, we have\begin{align}\label{aa2}\dfrac{\widetilde{\Omega}_{r}(z_{k})}{\sigma
^{s}(z_{k})(1-|z_{k}|^{2})^{2(s-1)}}& \le C\dfrac{\Omega(E(z,2r))}{\sigma^{s}(z)(1-|z|^{2})^{2s}}\notag\\& \asymp \dfrac{\widehat{\Omega}_{2r}(z)}{\sigma^{s}(z)(1-|z|^{2})^{2(s-1)}}\end{align}
whenever $z\in E(z_{k}, r)$. Therefore, we have
\begin{align*}\sum_{k=1}^{\infty}\dfrac{\widehat{\Omega}_{r}^{p}(z_{k})}
{\sigma^{sp}(z_{k})(1-|z_{k}|^{2})^{2(s-1)p-2}}&\le C\sum_{k=1}^{\infty}\int_{E(z_{k},r)}\dfrac{\widehat{\Omega}_{2r}^{p}(z)}{
\sigma^{sp}(z)(1-|z|^{2})^{2(s-1)p}}dA(z)\\ &\le CM\int_\mathbb{D}\dfrac{\widehat{\Omega}_{2r}^{p}(z)}{\sigma^{sp}(z)(1-|z|^{2})
^{2(s-1)p}}dA(z)
\end{align*}
Thus by Lemma  \ref{lem4}, we have $$\bigg\|\bigg\{\dfrac{\widehat{\Omega}_{r}(z_{k})}
{\sigma^{s}(z_{k})(1-|z_{k}|^{2})
^{2(s-1-\frac{1}{p})}}\bigg\}_{k}\bigg\|_{l^{p}}\le C\|\widehat{M}_{2r,s}\|_{p}\asymp\|\widehat{M}_{R,s}\|_{p}$$.

$(c) \Rightarrow (b)$. Finally, Suppose $\bigg\{\dfrac{\widehat{\Omega}_{r}(z_{k})}{\sigma^{s}(z_{k})(1-|z_{k}|^{2})^
{2(s-1-\frac{1}{p})}}\bigg\}_{k}\in {l^{p}}$ for some $r$-lattice $\{z_{k}\}_{k=1}^{\infty}$. Similar to (\ref{aa2}) $\widehat{\Omega}_{r}(z)\le C\widehat{\Omega}_{2r}(z_{k})$ for $z\in
E(z_{k},r)$. Therefore, we have
\begin{align*} \int_{\mathbb{D}}\bigg(\dfrac{\widehat{\Omega}_{r}
(\xi)}{\sigma^{s}(\xi)(1-|\xi|^{2})
^{2(s-1)}}\bigg)^{p}dA(\xi)
&\le \sum_{k=1}^{\infty}\int_{E(z_{k},r)}\dfrac{\widehat{\Omega}^{p}_{r}(\xi)}{
\sigma^{sp}(\xi)(1-|\xi|^{2})^{2(s-1)p}}dA(\xi)\\
&\le C\sum_{k=1}^{\infty}\dfrac{\widehat{\Omega}^{p}_{2r}(z_k)}{\sigma
^{sp}(z_k)(1-|z_k|^{2})^{2(s-1)p-2}}\\
&\le C\sum_{k=1}^{\infty}\dfrac{\widehat{\Omega}^{p}_{r}(z_k)}{\sigma
^{sp}(z_k)(1-|z_k|^{2})^{2(s-1)p-2}}
\end{align*}
 Above inequality and Lemma \ref{lem4}  implies that $$\|\widehat{M}_{R,s}\|_{p}\asymp\|\widehat{M}_{r,s}\|_{p}\le C\bigg\|\bigg\{\dfrac{\widehat{\Omega}_{r}(z_{k})}{\sigma^{s}(z_k)
(1-|z_{k}|^{2})^{2(s-1-\frac{1}{p})}}\bigg\}_{k}\bigg\|_{l^{p}}$$ for any $R\in (0,1)$.
\end{proof}
\section{Carleson measure characterizations}
\noindent In this section, we are using \textit{averaging function} and \textit{t-Berezin transform} as our main tools to characterize the $(p,q,\sigma)-$Bergman Carleson measure for $0<p,q<\infty$ and $t>0$. Let $\Omega $ be a finite positive Borel measure. Recall that,
\begin{itemize}
   \item $\Omega$ is a \textit{$(p,q,\sigma)-$Bergman Carleson measure} if the embedding $i:A_{\sigma}^{p}\longrightarrow L^{q}(\Omega)$ is bounded. In other words, we can say $\Omega$ is a $(p,q,\sigma)-$Bergman Carleson measure if there exists a finite constant $C>0$ such that $$\int_{\mathbb{D}}|f|^{q}d\Omega\le C\|f\|_{A_{\sigma}^{p}}^{q}$$ for all $f\in A_{\sigma}^{p}$.
   \item $\Omega$ is a \textit{vanishing $(p,q,\sigma)-$Bergman Carleson measure} if $\int_{\mathbb{D}}|f_{n}|^{q}d\Omega \longrightarrow0 $ as $n\longrightarrow \infty$ whenever $\{f_{n}\}$ is a bounded sequence in $A_{\sigma}^{p}$ which converges to $0$ uniformly on any compact subset of $\mathbb{D}$.
  \end{itemize}
Note that, by taking $p=q$ and $\sigma(z) =1$, $\Omega$ becomes a Bergman Carleson measure and vanishing Carleson measure. We divide our result into two cases: $0<p\le q<{\infty}$ and $0<q<p<{\infty}$.

\begin{thm} \label{1t} Let $\Omega$ be a finite positive Borel measure and $0<p\le q<{\infty}$. Then the following statements are equivalent:
\begin{enumerate}
  \item[{(a)}] $\Omega$ is a $(p,q,\sigma)-$Bergman Carleson measure.
  \item[{(b)}] The function $\dfrac{\widetilde{\Omega}_{t}(z)}{\sigma^{\frac{q}{p}}(z)(1-|z|^
      {2})^{2\frac{q}{p}-(\frac{\alpha +2}
{2})t}}$ is bounded on $\mathbb{D}$ for $t>\frac{2q}{p(\alpha+2)}$.
  \item[{(c)}] The function $\dfrac{\widehat{\Omega}_{R}(z)}{\sigma^{\frac{q}{p}}(z)(1-|z|^{2})^
{2(\frac{q-p}{p})}}$ is bounded on $\mathbb{D}$ for some (or any) $R\in (0,1)$.
  \item[{(d)}]The sequence $\bigg\{\dfrac{\widehat{\Omega}_{r}(a_{k})}{\sigma^{\frac{q}{p}
      }(z_k)(1-|z_k|^{2})^{2(\frac{q-p}{p})}}\bigg\}_{k=1}^{\infty}$ is bounded for some (or any) $r$-lattice $\{z_k\}_{k=1}^{\infty}$ with $0<r<1$.
  \end{enumerate}
Furthermore, \begin{align}\|i\|_{A_{\sigma}^{p}\rightarrow L^{p}(\Omega)}&\asymp \sup_{z\in \mathbb{D}}\dfrac{\widetilde{\Omega}_{t}(z)}{\sigma^{\frac{q}{p}}(z)(1-|z|^
      {2})^{2\frac{q}{p}-(\frac{\alpha +2}
{2})t}}\notag\\ &\asymp \sup_{z\in\mathbb{D}}\dfrac{\widehat{\Omega}_{R}(z)}{\sigma^{\frac{q}{p}}(z)
(1-|z|^{2})^{2(\frac{q-p}{p})}}\notag\\
&\asymp \sup_{k}\dfrac{\widehat{\Omega}_{r}(z_{k})}{\sigma^{\frac{q}{p}}(z_k)
(1-|z_k|^{2})
^{2(\frac{q-p}{p})}}\label{aas1}
\end{align}
\end{thm}
\begin{proof}$(d)\Rightarrow (a)$  We assume that $\{z_{k}\}_{k=1}^{\infty}$ be an $r$-lattice. We use the elementary inequality $\sum_{k=1}^{\infty}u_{k}^{l}\le (\sum_{k=1}^{\infty}u_k)^l$,\;\;\;$u_k\ge 0$, $k=1,2,\cdots$, by taking $l=\frac{p}{q}\ge1$, using Lemma  \ref{lem1}  and (\ref{a1}), we obtain
\begin{align} \int_{\mathbb{D}}|f(z)|^{q}d{\Omega} &= \sum_{k=1}^{\infty}\int_{E(z_k,r)}|f(z)|^{q}d{\Omega}\notag\\ &\le \sum_{k=1}^{\infty}\widehat{\Omega}_{r}(z_k)|E(z_k,r)|\bigg(\sup_{z\in E(z_k,r)}|f(z)|^{p}\bigg)^{\frac{q}{p}}\notag\\ &\le C\sum_{k=1}^{\infty}\dfrac{\widehat{\Omega}_{r}(z_{k})}{\sigma^{\frac
{q}{p}
}(z_k)(1-|z_k|^{2})^{2(\frac{q-p}{p})}}\bigg(\int_{E(z_k,\frac{1+r}{2})}|f(\xi)|
^{p}
\sigma(\xi)dA(\xi)\bigg)^\frac{q}{p}\notag\\
&\le C\sup_{k}\dfrac{\widehat{\Omega}_{r}(z_{k})}{\sigma^{\frac{q}{p}}(z_k)
(1-|z_k|^
{2})^{2(\frac{q-p}{p})}}\bigg(\sum_{k=1}^{\infty}\int_{E(z_k,\frac{1+r}
{2})}
|f(\xi)|^{p}\sigma(\xi)dA(\xi)\bigg)^\frac{q}{p}\notag\\
&\le CM^\frac{q}{p}\sup_{k}\dfrac{\widehat{\Omega}_{r}(z_{k})}{\sigma^
{\frac{q}
{p}}(z_k)(1-|z_k|^{2})^{2(\frac{q-p}{p})}}\|f\|_{p,\sigma}^{q}\label{aas2}
\end{align}
Above inequality reveals that \begin{equation*}\|i\|_{A_{\sigma}^{p}\rightarrow L^{q}(\Omega)}\le C\sup_{k}\dfrac{\widehat{\Omega}_{r}(z_{k})}{\sigma^{\frac{q}{p}}(z_k)
(1-|z_k|^{2})
^{2(\frac{q-p}{p})}}\end{equation*}.

$(a)\Rightarrow (c)$. Set $f_{z}(w)=K^{\alpha}(z,w),\;w\in \mathbb{D}$. By Lemma  \ref{lem1}  and statement$(a)$, we have
 \begin{align}\dfrac{\widehat{\Omega}_{R}(z)}{\sigma^{\frac{q}{p}}(z)
(1-|z|^{2})^{2(\frac{q-p}{p})}}
&=\dfrac{\Omega(E(z,R))}{\sigma^{\frac{q}{p}}(z)(1-|z|^{2})^{2\frac{q}{p}}}\notag\\
&\le C \dfrac{(1-|z|^{2})^{q(\alpha +2-\frac{2}{p})}}{\sigma^{\frac{q}{p}}(z)}
\int_{E(z,R)}|f_{z}(\xi)|^{q}d\Omega(\xi) \notag\\
&\le C \dfrac{(1-|z|^{2})^{q(\alpha +2-\frac{2}{p})}}{\sigma^{\frac{q}{p}}(z)}
\int_{\mathbb{D}}|f_{z}(\xi)|^{q}d\Omega(\xi)\label{aas3}\\
&\le C \dfrac{(1-|z|^{2})^{q(\alpha +2-\frac{2}{p})}}{\sigma^{\frac{q}{p}}(z)}\|i\|_{A_{\sigma}^{p}\rightarrow L^{q}(\Omega)}^{q}\|f_z\|_{A_{\sigma}^{p}}^{q}\notag\\
&\le C \|i\|_{A_{\sigma}^{p}\rightarrow L^{q}(\Omega)}^{q}\notag
 \end{align}
Above inequality reveals that \begin{equation}\label{aas4}\sup_{z\in \mathbb{D}}\dfrac{\widehat{\Omega}_{R}(z)}
{\sigma^{\frac{q}{p}}(z)(1-|z|
^{2})^{2(\frac{q-p}{p})}}\le \|i\|_{A_{\sigma}^{p}\rightarrow L^{q}(\Omega)}.\end{equation}
The equivalence of $(a)$,$(c)$ and $(d)$ follows from above proof of implications. Moreover,  \begin{align}\label{aas5}\|i\|_{A_{\sigma}^{p}\rightarrow L^{p}(\Omega)}&\asymp \sup_{z\in\mathbb{D}}\dfrac{\widehat{\Omega}_{R}(z)}{\sigma^{\frac{q}
{p}}(z)(1-|z|^{2})^{2(\frac{q-p}{p})}}\notag\\
&\asymp \sup_{k}\dfrac{\widehat{\Omega}_{r}(z_{k})}{\sigma^{\frac{q}{p}}(z_k)
(1-|z_k|^{2})^{2(\frac{q-p}{p})}}
\end{align}

$(b)\Rightarrow (c)$. For any $R$, $0<R<1$, Lemma \ref{lem7} yields \begin{equation}\label{aas6}\dfrac{\widehat{\Omega}_{R}(z)}{\sigma
^{\frac{q}{p}}(z)(1-|z|^{2})^{2(\frac{q-p}{p})}}\le C\dfrac{\widetilde{\Omega}_{t}(z)}{\sigma^{\frac{q}{p}}(z)(1-|z|^{2})
^{2\frac{q}{p}-(\frac{\alpha +2}{2})t}}\end{equation}.

$(a)\Rightarrow (b)$. The estimate \eqref{aas5} reveals that the embedding operator $i:A_{\sigma}^{p}\longrightarrow L^{q}(\Omega)$ is bounded for some $0<p\leq q<{\infty}$ if and only if $i:A_{\sigma}^{p_1}\longrightarrow L^{q_1}(\Omega)$ is bounded for some $0<p_1\leq q_1<{\infty}$ with $\frac{q_1}{p_1}=\frac{q}{p}$. Since $\Omega$ is a $(p,q,\sigma)$-carleson measure $i:A_{\sigma}^{Np}\longrightarrow L^{Nq}(\Omega)$ where $N$ is some integer with $Np>\frac{4}{(\alpha+2)}$. Let $f_z(.)=K_z^{\alpha}(.)$,\;\;$z\in\mathbb{D}$ and \eqref{aas5} tells us that \begin{align}\dfrac{\widetilde{\Omega}_{Nq}(z)}{\sigma^{\frac{q}{p}}(z)
(1-|z|^{2})^{2\frac{q}{p}-(\frac{\alpha +2}
{2})Nq}} &\asymp \dfrac{1}{\sigma^{\frac{q}{p}}(z)(1-|z|^{2})^{2\frac{q}{p}-(\frac{\alpha +2}
{2})Nq}}\int_{\mathbb{D}}|f_z(\xi)|^{Nq}d{\Omega}(\xi)\label{aa11}\\
&\lesssim \dfrac{\|i\|^{Nq}_{A_{\sigma}^{Np}\rightarrow L^{Nq}(\Omega)}\|f_{z}\|_{Np,\sigma}^{Nq}}{{\sigma^{\frac{q}{p}}(z)
(1-|z|^{2})^{2\frac{q}{p}-(\frac{\alpha +2}
{2})Nq}}}\notag\\&\le C\sup_{k}\dfrac{\widehat{\Omega}_{r}(z_{k})}{\sigma^{\frac{q}{p}}(z_k)
(1-|z_k|^{2})^{2(\frac{q-p}{p})}}\notag
\end{align}
Hence, for any $z\in \mathbb{D}$ \begin{align*}\dfrac{\widetilde{\Omega}_{Nq}(z)}{\sigma^{\frac{q}{p}}(z)
(1-|z|^{2})^{2\frac{q}{p}-(\frac{\alpha +2}{2})Nq}}
&\le C\|i\|^{Nq}_{A_{\sigma}^{Np}\rightarrow L^{Nq}(\Omega)}\\
&\asymp C\sup_{k}\dfrac{\widehat{\Omega}_{r}(z_{k})}{\sigma^{\frac{q}{p}}(z_k)
(1-|z_k|^{2})^{2(\frac{q-p}{p})}}
\end{align*}
Statement $(a)$ shows that the operator $i : A_{\sigma}^{\frac{tp}{Nq}}\rightarrow L^{\frac{t}{N}}(\Omega)$ is bounded.
Since $t>\frac{2q}{p(\alpha+2)}$, we have $\frac{tp(\alpha+2)}{q}>2$. The above calculations show that \begin{align*} \sup_{z\in \mathbb{D}}\dfrac{\widetilde{\Omega}_{t}(z)}
{\sigma^{\frac{q}{p}}(z)(1-|z|
^{2})^{2\frac{q}{p}-(\frac{\alpha +2}
{2})t}}&\le C\|i\|^{Nq}_{A_{\sigma}^{Np}\rightarrow L^{Nq}(\Omega)}\\ &\asymp C\sup_{k}\dfrac{\widehat{\Omega}_{r}(z_{k})}
{\sigma^{\frac{q}{p}}(z_k)
(1-|z_k|^{2})^{2(\frac{q-p}{p})}}
\end{align*}
This completes the proof.
\end{proof}
\begin{cor}\label{cor:continuous}  Let $t>0$, $0<p<{\infty}$, $d\Omega=\sigma dA$ is a positive measure and ${C}_{\psi,\varphi} : A_{\sigma}^{p}\to L^{p}(\Omega)$ be a bounded operator. Then the following statements are equivalent.
\begin{enumerate}
\item [{(a)}] ${C}_{\psi,\varphi} $ is power bounded, that is,
 $$Q_1 = \sup_{n \in \mathbb N}\|C^n_{\psi,\varphi}\|^p < \infty. $$
\item [{(b)}] The sequence of functions $\{f_n\}_{n = 1}^\infty$ defined on $\mathbb D$ as  $$f_n(z) = \frac{\widetilde{{\Omega}}_{n,t}(z)}{\sigma(z) \delta^{2- ({\alpha +2})t/
{2}}(z)}, \; z \in \mathbb D$$ is a norm bounded family  in $L^{\infty}(\mathbb{D})$ for $t >  {2}/{(\alpha+2)}$, that is,
 $$Q_2 = \sup_{n \in \mathbb N}\|f_n\|_\infty = {\rm sup}_{n \in \mathbb N} {\rm ess\: sup}_{z \in \mathbb D}\frac{\widetilde{{\Omega}}_{n,t}(z)}{\sigma(z) \delta^{2- ({\alpha +2})t/
{2}}(z)} < \infty. $$
\item [{(c)}] The sequence of functions $\{g_n\}_{n = 1}^\infty$ defined on $\mathbb D$ as  $$g_n(z)  = \frac{\widehat{\Omega}_{n,R}(z)}{\sigma(z)}, \; z \in \mathbb D $$ is  a norm bounded family   in $L^{\infty}(\mathbb{D})$ for some (or any) $R\in (0,1)$, that is,
 $$Q_3 = \sup_{n \in \mathbb N}\|g_n\|_\infty = {\rm sup}_{n \in \mathbb N} {\rm ess\: sup}_{z \in \mathbb D}\frac{\widehat{\Omega}_{n,R}(z)}{\sigma(z)} < \infty. $$
\item [{(d)}] The double sequence $\gamma_{nk}=\{\gamma_{n, k}\}_{n, k}$, where $$\gamma_{n, k} = \frac{\widehat{\Omega}_{n,r}(z_k)}{\sigma(z_k)}   $$ is  bounded for some {\rm(}any{\rm)} $r$-lattice $\{z_k\}_{k = 1}^{\infty}$ with fixed $r$, $0<r<1$, that is,
 $$Q_4 = \|\gamma_{nk}\|_{\Lambda^2_\infty} = \sup_{n, k \in \mathbb N}\frac{\widehat{\Omega}_{n,r}(z_k)}{\sigma(z_k)}  < \infty. $$   \end{enumerate}
Moreover,
$ \label{l1} Q_1 \asymp Q_2 \asymp Q_3 \asymp Q_4.$
\end{cor}

\begin{thm}\label{2t} Let $\Omega \ge 0$ and $0<p\le q<{\infty}$. Then the following statement are equivalent:
\begin{enumerate}
  \item[{(a)}] $\Omega$ is a vanishing $(p,q,\sigma)$-Bergman Carleson measure.
    \item [{(b)}]For $t>\frac{2q}{p(\alpha+2)}$, we have $$\dfrac{\widetilde{\Omega}_{t}(z)}{\sigma^{\frac{q}{p}}(z)(1-|z|^
      {2})^{2\frac{q}{p}-(\frac{\alpha +2}
{2})t}}\longrightarrow 0$$  as $z\longrightarrow\partial\mathbb{D}$.
    \item [{(c)}]For some $($or any$)$ $R\in(0,1)$, we have $$\dfrac{\widehat{\Omega}_{R}(z)}{\sigma^{\frac{q}{p}}(z)(1-|z|^{2
  })^{2(\frac{q-p}{p})}}\longrightarrow 0$$  as $z\longrightarrow\partial\mathbb{D}$.
    \item [{(d)}]For some (or any) $r$-lattice $\{z_k\}_{k=1}^{\infty}$ with $r\in(0,1)$, we have $$\dfrac{\widehat{\Omega}_{r}(z_{k})}
        {\sigma^{\frac{q}{p}}(z_k)(1-|z_k|^
        {2})^{2(\frac{q-p}{p})}}\longrightarrow 0$$ as $k\longrightarrow 0$.
  \end{enumerate}

\end{thm}
\begin{proof}
The implication $(c)\Rightarrow(d)$ is trivial because $z_{k}\rightarrow\partial\mathbb{D}$ as $k\rightarrow{\infty}$ whenever $\{z_k\}_{k=1}^{\infty}$ is an $r$-lattice. It follows from \eqref{aas6} that $(b)\Rightarrow (c)$.\\
$(a)\Rightarrow (c)$ Given $0<R<1$. For $z\in\mathbb{D}$, we set $f_{z}(\xi)=\dfrac{K_z^{\alpha}(\xi)}{\sigma^{\frac{1}{p}}(z)(1-|z|^{2})^{\frac
{2}{p}-(\alpha +2)}}$,\;\;$\xi\in\mathbb{D}$. One can easily find that $f_z\in A_{\sigma}^{p}$, $\|f_{z}\|_{p,\sigma}\le C$ and $f_{z}\rightarrow 0$ uniformly on any compact subset of $\mathbb{D}$ as $z\rightarrow \partial\mathbb{D}$. Since $\Omega$ is a vanishing $(p,q,\sigma)-$Bergman Carleson measure, it follows from \eqref{aas3} that $$\dfrac{\widehat{\Omega}_{R}(z)}{\sigma^
{\frac{q}{p}}(z)(1-|z|^{2})^{2(\frac{q-p}{p})}}\le C \dfrac{(1-|z|^{2})^{q(\alpha +2-\frac{2}{p})}}{\sigma^{\frac{q}{p}}(z)}
\int_{\mathbb{D}}|f_{z}(\xi)|^{q}d\Omega(\xi)\longrightarrow 0 $$ as $z\rightarrow\partial\mathbb{D}.$\\

$(d)\Rightarrow (a)$. Suppose $(d)$ holds. For any $\epsilon>0$, there exists a positive integer $k_0$ such that
$\dfrac{\widehat{\Omega}_{r}(z_{k})}{\sigma^
{\frac{q}{p}}(z_k)(1-|z_k|^{2})^{2(\frac{q-p}{p})}}< \epsilon$, whenever $k>k_0$. Notice that $\cup_{k=1}^{k_0}E(z_k,r)$ is relatively compact in $\mathbb{D}$. Let us consider a bounded sequence $\{f_j\}_{j=1}^{\infty}$ in $A_{\sigma}^{p}$ such that $f_j\longrightarrow 0$ uniformly on any compact subset of $\mathbb{D}$  as $j\rightarrow{\infty}$. Similar to the proof of \eqref{aas2} and if $j$ is large enough, we have that
\begin{align}\label{aas7}\int_{\mathbb{D}}|f_{j}(z)|^{q}d\Omega(z)&\le \int_{\cup_{k=1}^{k_0}E(z_k,r)}|f_{j}(z)|^{q}d\Omega(z)+\sum_{k=k_{0}+1}
^{\infty}\int_{E(z_k,r)}|f_{j}(z)|^{q}d\Omega(z)\notag\\&\le C\epsilon\|f_j\|_{p,\sigma}^{q}\notag\\&\le C\epsilon.
\end{align} where $C$ is independent of $\epsilon$.\\

 $(a)\Rightarrow (b)$ The equivalence of $(a)$, $(c)$ and $(d)$ shows that the measure $\Omega$ is a vanishing $(Np,Nq,\sigma)$-Bergman Carleson measure if $\Omega$ is a vanishing $(p,q,\sigma)$-Bergman Carleson measure. For $z\in \mathbb{D}$, set $f_{z}(\xi)=\dfrac{K_z^{\alpha}(\xi)}{\sigma^
 {\frac{1}{p}}(z)(1-|z|^{2})^{\frac
{2}{p}-(\alpha +2)}}$,\;\;$\xi\in\mathbb{D}$. One can easily find that $f_z\in A_{\sigma}^{p}$, $\|f_{z}\|_{p,\sigma}\le C$ and $f_{z}\rightarrow 0$ uniformly on any compact subset of $\mathbb{D}$ as $z\rightarrow\partial\mathbb{D}$. Since $Np>\frac{4}{(\alpha+2)}$ and it follows from \eqref{aa11}, we have that $$\dfrac{\widetilde{\Omega}_{Nq}(z)}{\sigma^{\frac{q}{p}}(z)
(1-|z|^{2})^{2\frac{q}{p}-(\frac{\alpha +2}{2})Nq}}
\asymp \dfrac{1}{\sigma^{\frac{q}{p}}(z)(1-|z|^{2})^{2\frac{q}{p}-(\frac{\alpha +2}
{2})Nq}}\int_{\mathbb{D}}|f_z(\xi)|^{Nq}d{\Omega}(\xi).$$ Statement $(a)$ yields that $\Omega$ is a vanishing $(\frac{tp}{Nq},\frac{t}{N},\sigma)-$Bergman Carleson measure. Therefore $$\lim_{z\rightarrow\partial\mathbb{D}}\dfrac{\widetilde
{\Omega}_{Nq}(z)}{\sigma^{\frac{q}{p}}(z)
(1-|z|^{2})^{2\frac{q}{p}-(\frac{\alpha +2}{2})Nq}}=0.$$ The proof is completed.
 \end{proof}
\begin{cor}\label{cor:continuous}  Let $t>0$, $0<p<{\infty}$, $d\Omega=\sigma dA$ is a positive measure and ${C}_{\psi,\varphi} : A_{\sigma}^{p}\to L^{p}(\Omega)$ be a bounded operator. Then the following statements are equivalent.
\begin{enumerate}
\item [{(a)}] ${C}_{\psi,\varphi} $ is power compact, that is, ${C}^{n}_{\psi,\varphi} $ is compact, for some $n\in \NN$.
    \item [{(b)}]For $t>\frac{2q}{p(\alpha+2)}$, we have $$\dfrac{\widetilde{\Omega}_{n,t}(z)}{\sigma^{\frac{q}{p}}(z)(1-|z|^
      {2})^{2\frac{q}{p}-(\frac{\alpha +2}
{2})t}}\longrightarrow 0$$  as $z\longrightarrow\partial\mathbb{D}$.
    \item [{(c)}]For some $($or any$)$ $R\in(0,1)$, we have $$\dfrac{\widehat{\Omega}_{n,R}(z)}{\sigma^{\frac{q}{p}}(z)(1-|z|^{2
  })^{2(\frac{q-p}{p})}}\longrightarrow 0$$  as $z\longrightarrow\partial\mathbb{D}$.
    \item [{(d)}]For some (or any) $r$-lattice $\{z_k\}_{k=1}^{\infty}$ with $r\in(0,1)$, we have $$\dfrac{\widehat{\Omega}_{n,r}(z_{k})}
        {\sigma^{\frac{q}{p}}(z_k)(1-|z_k|^
        {2})^{2(\frac{q-p}{p})}}\longrightarrow 0$$ as $k\longrightarrow 0$.
  \end{enumerate}
\end{cor}

\begin{thm}\label{3t}  Let $\Omega$ be a finite positive Borel measure and $0<q<p<{\infty}$. Then the following statements are equivalent:
\begin{enumerate}
  \item[{(a)}] $\Omega$ is a $(p,q,\sigma)-$Bergman Carleson measure.
  \item[{(b)}] $\Omega$ is a vanishing $(p,q,\sigma)-$Bergman Carleson measure.
\item [{(c)}]For $t>\frac{2(q+p)}{p(\alpha+2)}$, we have $$\widetilde{M}_{t}(z)=\dfrac{\widetilde
    {\Omega}_{t}(z)}{\sigma^{\frac{q}{p}}(z)(1-|z|^
      {2})^{2\frac{q}{p}+1-(\frac{\alpha +2}{2})t}} \in L^{\frac{p}{p-q}}(\mathbb{D}).$$
    \item [{(d)}]For some $($or any$)$ $R\in (0,1)$, we have $$\widehat{M}_{R}(z)=\dfrac{\widehat
        {\Omega}_{R}(z)}{\sigma^{\frac{q}{p}}(z)}\in L^{\frac{p}{p-q}}(\mathbb{D}). $$
    \item [{(e)}]For some (or any) $r$-lattice $\{z_k\}_{k=1}^{p}$ with $r\in(0,1)$, we have $$\dfrac{\widehat{\Omega}_{r}(z_{k})}
        {\sigma^{\frac{q}{p}}(z_k)(1-|z_k|^{2})^{2(\frac{q-p}{p})}} \in l^{\frac{p}{p-q}}(\mathbb{D}).$$
  \end{enumerate}
 Moreover, we have \begin{align}\label{aas8} \|i\|_{A_{\sigma}^{p}\rightarrow L^{p}(\Omega)}^{q}&\asymp \|\widetilde{M}_{t}\|_{\frac{p}{p-q}}&\asymp
 \|\widehat{M}_{R}\|_{\frac{p}{p-q}} &\asymp \bigg\|\bigg\{\dfrac{\widehat{\Omega}_{r}(z_{k})}{\sigma^{\frac{q}{p}}(z_k)
(1-|z_{k}|^{2})^{2(\frac{q-p}{p})}}\bigg\}_{k=1}
^{\infty}\bigg\|_{l^{\frac{p}{p-q}}}.\end{align}
\end{thm}
\begin{proof}
 Since Lemma  \ref{lem7}  implies that the statements $(c)$, $(d)$ and $(e)$ are equivalent with the corresponding norm estimate (\ref{aas8}) and the implication $(b)\Rightarrow (a)$ is trivially true, it is sufficient to prove that $(d)\Rightarrow (a)$, $(a)\Rightarrow (e)$ and $(a)\Rightarrow (b)$.\\
$(d)\Rightarrow (a)$. Since $0<q<p<{\infty}$ implies $\frac{p}{q}>1$ ,therefore the conjugate exponent of $\frac{p}{q}$ is $\frac{p}{p-q}$. For $f\in A_{\sigma}^{p}$, we have
\begin{align*}\int_{\mathbb{D}}|f(\xi)|^{q}d\Omega(\xi)&\le C\int_{\mathbb{D}}|f(\xi)|^{q}\widehat{\Omega}_{R}(\xi)dA(\xi)\\&\le C\bigg\|\widehat{M}_{R}\bigg\|_{\frac{p}{p-q}}\|f\|_{p,\sigma}^{q}.
\end{align*} Above inequality follows from Lemma \ref{lem3}  and Holder's inequality shows that $\Omega$ is a $(p,q,\sigma)-$Bergman Carleson measure and $\|i\|_{A_{\sigma}^{p}\rightarrow L^{p}(\Omega)}^{q}\leq C\|\widehat{M}_{R}\|{\frac{p}{p-q}}$.\\
$(a)\Rightarrow (e)$. Let $\{\lambda_k\}_{k=1}^{\infty}\in l^{p}$ and set $f$ as in  Lemma  \ref{lem6}. Statement $(a)$ and Lemma \ref{lem6}  implies that \begin{equation}\label{e1}\int_{\mathbb{D}}\bigg|\sum_{k=1}^{\infty}
\lambda_{k}\dfrac{K_{z_k}^{\alpha}(z)}{\sigma^
{\frac{1}{p}}(z_k)(1-|z_k|^{2})^{\frac{2}{p}-(\alpha +2)p}}\bigg|^{q}d\Omega(z)\le C\|i\|_{A_{\sigma}^{p}\rightarrow L^{q}}^{q}\|\{\lambda_k\}_{k}\|_{l^p}^{q}.\end{equation}
Recall that Rademacher functions $\psi_k$ are defined by
 \[\psi_{0}(t) = \left\{
              \begin{array}{ll} \; 1,\; & \text{ if } 0\leq t-[t]<{1}/{2}  \\
                  -1,\; & \text{ if }    {1}/{2}\leq  t-[t]<1              \end{array}
       \right.
\] and $\psi_k(t)=\psi_{0}(2^{k}t)$ for $k=1,2,\cdots$, where $[t]$ denotes the greatest integer less than or equal to $t$. For $0<q<\infty$, Khinchine's inequality is given as
 $$C_1 \bigg( \displaystyle \sum_{k=1}^{m}|b_k|^2\bigg)^{ {q}/{2}}\leq \int_{0}^{1}\bigg|\displaystyle\sum_{k=1}^{m}b_k\psi_{k}(t)\bigg|^{q}dt\leq C_2\bigg(\sum_{k=1}^{m}|b_k|^{2}\bigg)^{ {q}/{2}}, $$
which  holds  for all $m\geq1$ and all complex numbers $b_1,b_2,\cdots b_m$.
Let $\psi_k(t)$ be the $k$th Rademacher function on $[0,1]$. Replacing $\lambda_k$ with $\psi_k(t)\lambda_k$, integrating w.r.t $t$ from $0$ to $1$ and applying Khinchine's inequality in [\ref{e1}], we see that $$\int_{\mathbb{D}}\bigg(\sum_{k=1}^{\infty}|\lambda_{k}|^{2}
\dfrac{|K_{z_k}^{\alpha}(z)|^{2}}{\sigma^
{\frac{2}{p}}(z_k)(1-|z_k|^{2})^{2(\frac{2}{p}-(\alpha +2))}}\bigg)^{\frac{q}{2}}d\Omega(z)\le C\|i\|_{A_{\sigma}^{p}\rightarrow L^{q}}^{q}\|\{\lambda_k\}_{k}\|_{l^p}^{q}.$$
Thus, we have
\begin{align}\label{aas9}
&\sum_{k>k_0}|\lambda_k|^{q}\dfrac{\widehat
{\Omega}_{r}(z_{k})}{\sigma^{\frac{q}{p}}(z_k)(1-|z_k|^{2})^{2(\frac{q-p}{p})}}\notag\\ &\asymp \sum_{k>k_0}\int_{E(z_k,r)}|\lambda_k|^{q}|K_{z_k}^{\alpha}(z)|^{q}\sigma ^{\frac{q}{p}}(z)(1-|z|^{2})^{(\alpha+2)q-\frac{2q}{p}}d\Omega(z)\notag\\ &\le C\sum_{k>k_0}\int_{E(z_k,r)}\bigg(\sum_{k=1}^{\infty}|\lambda_{k}|^{2}\dfrac
{|K_{z_k}^{\alpha}(z)|^{2}}{\sigma^{\frac{2}{p}}(z_k)(1-|z_k|^{2})^{2(\frac{2}{p}-(\alpha +2))}}\bigg)^{\frac{q}{2}}d\Omega(z)\notag\\
&\le C\int_{\mathbb{D}}\bigg(\sum_{k=1}^{\infty}|\lambda_{k}|^{2} \dfrac{|K_{z_k}^{\alpha}(z)|^{2}}{\sigma^{\frac{2}{p}}(z_k)(1-|z_k|^{2})^{2(\frac{2}{p}-(\alpha +2))}}\bigg)^{\frac{q}{2}}d\Omega(z)\notag\\
&\le C\|i\|_{A_{\sigma}^{p}\rightarrow L^{q}}^{q}\|\{\lambda_k\}_{k}\|_{l^p}^{q}.
\end{align}
Since $\cup_{k=1}^{k_0}E(z_k,R)$ is relatively compact in $\mathbb{D}$, so $$\sum_{k=0}^{\infty}|\lambda_k|^{q}\dfrac{\widehat
{\Omega}_{r}(z_{k})}
{\sigma^{\frac{q}{p}}(z_k)(1-|z_k|^{2})^{2(\frac{q-p}{p})}}
\le C\|i\|_{A_{\sigma}^{p}\rightarrow L^{q}}^{q}\|\{\lambda_k\}_{k}\|_{l^p}^{q}.$$
Setting $c_k=|\lambda_k|^{q}$, for each $k$, then $\{c_k\}_{k=1}^{\infty}\in l^{\frac{p}{q}}$ because $\{\lambda_k\}_{k=1}^{\infty}\in l^{p}$ implies that $$\sum_{k=0}^{\infty}c_k\dfrac{\widehat{\Omega}_{r}
(z_{k})}{\sigma^{\frac{q}{p}}(z_k)(1-|z_k|^{2})^{2(\frac{q-p}{p})}}
\le C\|i\|_{A_{\sigma}^{p}\rightarrow L^{q}}^{q}\|\{c_k\}\|_{l^\frac{p}{q}}.$$ Hence by duality argument,we have $$\bigg\|\bigg\{\dfrac{\widehat{\Omega}_{r}(z_{k})}
{\sigma^{\frac{q}{p}}(z_k)(1-|z_{k}|^{2})^{2(\frac{q-p}{p})}}\bigg\}_{k=1}
^{\infty}\bigg\|_{l^{\frac{p}{p-q}}}\le C \|i\|_{A_{\sigma}^{p}\rightarrow L^{p}(\Omega)}^{q}.$$
Finally, we will prove the implication $(a)\Rightarrow (b)$. Let us consider a bounded sequence $\{f_{n}\}_{n=1}^{\infty}$ in $A_{\sigma}^{p}$ such that $f_{n}\longrightarrow 0$ uniformly on each compact subset of $\mathbb{D}$. Let $F$ be any compact subset of $\mathbb{D}$ and $\Omega_F$ be the restriction of $\Omega$ to $F$. Then we have
\begin{align}\label{aas10}
 \int_{\mathbb{D}}|f_n(z)|^{q}d\Omega(z) &=\int_{F}+\int_{\mathbb{D}\setminus F}|f_n(z)|^{q}d\Omega(z)\notag\\ &=I_1+I_2.
\end{align}
Since $f_{n}\rightarrow 0$ uniformly on $F$ as $n\rightarrow \infty$, we have \begin{align*}I_1&=\int_{F}|f_n(z)|^{q}d\Omega(z)\\
&\le C\sup_{z\in F}|f_n(z)|^{q}\longrightarrow 0.
\end{align*} Fix $R\in (0,1)$ and $z\in \mathbb{D}$, and we obtain that $$\dfrac{\widehat{(\Omega_F)}_R(z)}{\sigma^{\frac{q}{p}}(z)}\longrightarrow 0$$ as $F$ extended to $\mathbb{D}$. By the equivalence of $(a)$ and $(d)$, we have $$\dfrac{\widehat{(\Omega_F)}_R^\frac{p}{p-q}(z)}{\sigma^{\frac{q}{p-q}}(z)}\le \dfrac{\widehat{(\Omega_F)}_R^{\frac{p}{p-q}}(z)}{\sigma^{\frac{q}{p-q}}(z)}\in L^{1}(\mathbb{D}).$$
Therefore,
\begin{align}\label{aas11}
 I_2&=\int_{\mathbb{D}}|f_n(z)|^{q}d\Omega_{F}(z)\notag\\&\le C\sup_{n}\|f_n\|_{p,\sigma}^{q}\bigg\|\dfrac{\widehat{(\Omega_F)
 }_R(z)}{\sigma^{\frac{q}{p}}(z)}\bigg\|_{\frac{p}{p-q}}\notag\\&\longrightarrow 0\;\;\; \text{as}\; F \;\text{extended to} \;\mathbb{D},
\end{align} above follows from (\ref{a1}) and the dominated convergence theorem. Hence $$\lim_{n\rightarrow \infty}  \int_{\mathbb{D}}|f_n(z)|^{q}d\Omega(z)=0$$ implies that $\Omega$ is a vanishing $(p,q,\sigma)-$Bergman Carleson measure. This completes the proof.
\end{proof}

\end{document}